\documentclass[a4paper, 11pt, reqno]{amsart}
\usepackage[usenames,dvipsnames]{color}
\usepackage{amssymb, amsmath, latexsym, graphics, graphicx}
\usepackage{ytableau}
\usepackage{mathrsfs}

\newtheorem{theorem}{Theorem}[section]
\newtheorem{lemma}[theorem]{Lemma}
\newtheorem{corollary}[theorem]{Corollary}

\newtheorem{question}[theorem]{Question}

\newtheorem{conjecture}[theorem]{Conjecture}

\theoremstyle{definition}
\newtheorem{example}[theorem]{Example}
\newtheorem{remark}[theorem]{Remark}

\allowdisplaybreaks

\makeatother

\newcommand\trop{\mathbb{T}}
\newcommand{\plac}[1]{\mathbb{P}_#1}

\newcommand{\twon}{{2^{[n]}}}
\newcommand{\utst}{M_{\twon}(\trop)}
\newcommand{\bark}{{\hat{[k]}}}
\numberwithin{equation}{section}

\begin{document}
\title[Tropical Representations and Identities of Plactic Monoids]{Tropical Representations and Identities of Plactic Monoids}

\date{\today}
\thanks{}

\maketitle

\begin{center}

MARIANNE JOHNSON\footnote{Email \texttt{Marianne.Johnson@manchester.ac.uk}} and
MARK KAMBITES\footnote{Email \texttt{Mark.Kambites@manchester.ac.uk}}.

\medskip

Department of Mathematics, University of Manchester, \\
Manchester M13 9PL, UK.

\end{center}

\begin{abstract} We exhibit a faithful representation of the plactic monoid of every finite rank as a monoid of upper triangular matrices over the tropical semiring. This answers a question first posed by Izhakian and subsequently studied by several authors. A consequence is a proof of a conjecture of Kubat and Okni\'{n}ski that every plactic monoid of finite rank satisfies a non-trivial semigroup identity. In the converse direction, we show that every
identity satisfied by the plactic monoid of rank $n$ is satisfied by the monoid of $n \times n$ upper triangular tropical matrices. In particular this implies that the variety generated by the $3 \times 3$ upper triangular tropical matrices coincides with that generated by the plactic monoid of rank $3$, answering another question
of Izhakian.
\end{abstract}

An important family of monoids, which have attracted much attention due to their interesting combinatorics and applications in representation theory, are the \textit{plactic monoids}. These monoids arise from the combinatorics of tableaux by identifying words over a fixed ordered alphabet whenever they produce the same tableau via Schensted's insertion algorithm \cite{Schensted}. Knuth \cite{Knuth} gave a presentation of the plactic monoids in terms of certain balanced relations of length 3, and these monoids were later studied in detail by Lascoux and Sch\"{u}tzenberger \cite{Lascoux81}. The plactic monoids have applications in algebraic combinatorics and representation theory (due to the important role played by Young tableaux in the representation of the symmetric and linear groups); they have been used to prove the Littlewood-Richardson rule \cite{Schutzenberger77} (a combinatorial rule in the theory of symmetric functions describing the product of Schur functions, or equivalently, in representation theory describing certain tensor products of irreducible representations of unitary groups) and to provide a combinatorial description of the Kostka-Foulkes polynomials \cite{Lascoux78} (which arise as entries of the character tables of finite linear groups).   Recent research has also focused on the monoid algebras of the plactic monoids: in the rank 3 case, Kubat and Okni\'{n}ski described the prime ideals and irreducible representations \cite{Kubat12} and constructed a finite Gr\"{o}bner-Shorshov basis \cite{Kubat14} for each plactic algebra. In the rank 4 case, Ced\'{o}, Kubat and Okni\'{n}ski studied the irreducible representations.  Cain, Gray and Malheiro \cite{Gray15} constructed a finite complete rewriting system for each plactic monoid and used this to prove that every plactic algebra of finite rank admits a finite Gr\"{o}bner-Shirshov basis, and that each plactic monoid of finite rank is biautomatic; they recently generalised these results to other \emph{crystal monoids} \cite{Gray19} (another description of the plactic monoids being in terms of \emph{crystal bases} in the sense of  Kashiwara \cite{Kashiwara}).  

The \textit{tropical semiring} (also known as the \textit{max-plus semiring}) is widely studied due its applications in numerous areas such as algebraic 
geometry and combinatorial optimisation (see just for example \cite{Akian, Allamigeon, Butkovic, Draisma12, GG16, MaclaganSturmfels, Mikhalkin}). The structure of tropical matrices plays a natural role in many of these applications, and has recently been the subject of extensive study (for example \cite{Merlet15, HK12, IJK16, JK13, Shitov14}). From an algebraic perspective, the tropical semiring is of interest as a natural carrier for representations of semigroups, including important infinite semigroups which do
not admit faithful finite dimensional representations over fields. Perhaps the best example is the \textit{bicyclic monoid}, which is ubiquitous
in infinite semigroup theory, appearing as a submonoid of numerous important semigroups. This monoid admits no faithful finite dimensional representation over any field, posing an obstruction to representing the many semigroups containing it, but it does have faithful representations over the tropical semiring \cite{Daviaud18,Izhakian10}.

Izhakian \cite[Theorem 7.17]{Izhakian19} first showed that the plactic monoid of rank $3$ admits a faithful representation by $6 \times 6$ upper triangular tropical matrices. A completely different description of what turns out to be essentially the same representation (see Remarks~\ref{remark_izhakiancain} and \ref{remark_bigblockcain} below) was given by Cain \textit{et al} \cite{Cain17}. This representation admits natural analogues in higher rank, but these are not faithful (see \cite[example before
Theorem 7.18]{Izhakian19} and Remark~\ref{remark_cainnotfaithful} below). Izhakian \cite[Problem 8.1]{Izhakian19} asked whether plactic
monoids of rank $4$ and above admit faithful representations by tropical matrices (upper triangular or otherwise); this question was also remarked upon in \cite{Okninski19}. One of the main aims of this paper is to answer this question in the positive. We shall exhibit an explicit faithful representation of the plactic monoid of every rank within a \textit{chain-structured tropical matrix semigroup}; these semigroups, which were introduced in \cite{Daviaud18}, are in particular examples of upper triangular tropical matrix semigroups.

A related question is that of whether the plactic monoids satisfy non-trivial semigroup identities. Kubat and Okni\'{n}ski \cite[Theorem 2.6]{Kubat15} showed that the plactic monoid of rank $3$ satisfies a non-trivial semigroup identity, and conjectured that the
same is true for every plactic monoid of finite rank. Cain and Malheiro \cite{Cain16} answered the corresponding
question in several classes of monoids related to plactic monoids and, in the light of this new evidence, repeated the original conjecture
about plactic monoids; Cain, Malheiro and Silva \cite{Cain17b} later answered the question for further classes of related monoids.
 On the other hand, Cain \textit{et al} \cite[Proposition 3.1]{Cain17} showed that the plactic monoid of \textit{countably infinite} rank does not satisfy any non-trivial semigroup identity; this implies that no \textit{single} identity is satisfied by plactic monoids of \textit{every} finite rank. A proposed proof of the original conjecture, using a mixture of a (non-faithful) tropical representation and more purely combinatorial techniques appeared in a recent preprint of Okni\'{n}ski \cite{Okninski19} but this has been retracted due to an incorrect proof. Since it is known that the upper triangular tropical matrices of each rank satisfy semigroup identities \cite{Izhakian14,Okninski15,Taylor17} (see also \cite{Daviaud18} for more discussion and \cite{Izhakian18} for the more general non-upper-triangular case), our faithful representations of plactic monoids by
tropical matrices lead to a proof of the conjecture.

There are also interesting questions about identities in the converse direction. The plactic monoid of rank $2$ is quite easily seen (see Remark~\ref{remark_rank2} below) to satisfy exactly the same identities as the bicyclic monoid, or equivalently \cite[Theorem 4.1]{Daviaud18} the monoid of $2 \times 2$ upper triangular tropical matrices. Izhakian \cite[Corollary 7.19]{Izhakian19} has shown that the plactic monoid of rank $3$ satisfies every identity satisfied by the monoid of $3\times 3$ upper triangular tropical matrices, and posed the question of whether the converse holds. We answer this question in the positive, showing that the variety generated by the plactic
monoid of rank $3$ coincides with that generated by the $3 \times 3$ upper triangular tropical matrices. More generally, we establish that the variety generated by the plactic monoid of rank $n$ is bounded below by the $n \times n$ upper triangular tropical matrix variety, and bounded above by the $d \times d$ upper triangular tropical matrix variety where $d$ is the integer part of $\frac{n^2}{4} +1$. We pose some questions
about whether these bounds can be improved.

In addition to this introduction, the paper comprises four sections. Section~\ref{sec_prelim} recaps relevant definitions, including those of the tropical semiring and the plactic
 monoids, and basic facts about them including the relationship between plactic monoid elements and (semi-standard) tableaux. Section~\ref{sec_rep} constructs a faithful representation of each plactic monoid by upper triangular tropical matrices. Section~\ref{sec_tropicaltoplactic} combines the results of Section~\ref{sec_rep} with some results from \cite{Daviaud18} to deduce results about identities satisfied by plactic monoids. Conversely, Section~\ref{sec_plactictotropical} uses in more detail the theory developed in \cite{Daviaud18} to prove that every identity satisfied in $\plac{n}$ is satisfied by the semigroup of all $n \times n$ upper triangular tropical matrices.

\section{Preliminaries}\label{sec_prelim}
We briefly recall the necessary definitions and notation. For further background on the plactic monoid and the combinatorics of tableaux we refer the reader to \cite{Fulton} and \cite{Lothaire}. For further information on the tropical semiring, semigroup-theoretic properties tropical matrices, and applications of these in semigroup theory we refer the reader to the now extensive literature in this area, including for example \cite{IJK16, IJK18, JK13, Pin98, Simon94}.

For each natural number $n$, let $[n]$ denote the set of integers from $1$ to $n$. For natural numbers $k \leq m$ we write $[k,m]$ for the
set of integers from $k$ to $n$ inclusive, and 
$\bark_m$ for the set of the $k$ largest integers in $[m]$, that is, $[m-k+1, m]$. We write $[n]^*$ for the free monoid on $[n]$, that is, the set of (possibly
empty) finite sequences of elements from $[n]$ (which we term \textit{words}) equipped with the operation of concatenation.  For $w \in [n]^*$ we write $|w|$ for the length of
$w$, and $w_i$ for the $i$th letter of $w$. We write $\epsilon$ for the empty word.

By a \textit{scattered subword} of $w \in [n]^*$ we formally mean a (possibly empty) strictly increasing finite sequence of integers from $[|w|]$,
considered as indices identifying letters in the word $w$ which compose the subword. However, for brevity we shall often identify the sequence $i_1 < \cdots < i_k$ with the word $w_{i_1} \dots w_{i_k}$. In particular, we shall sometimes say that a scattered subword is (for example) \textit{strictly descending}, by which we mean that the word $w_{i_1} \dots w_{i_k}$ (rather than the sequence $i_1, \dots i_k$ itself) is a strictly descending sequence.

For $n \geq 1$, the \textit{plactic monoid of rank $n$} is most easily defined as the monoid generated by the set $[n]$ subject to the \emph{Knuth relations}:
\begin{eqnarray*}
bca&=&bac \mbox{ for all } 1 \leq a <b \leq c \leq n; \textrm{ and } \\
cab&=&acb \mbox{ for all } 1 \leq a \leq b < c \leq n.
\end{eqnarray*}
We shall denote this monoid by $\plac{n}$. Plactic monoids admit the following alternative combinatorial description. A \textit{Young diagram} is a finite, left-aligned array of identically-sized boxes, with at least as many boxes in each row as in the row above\footnote{ Note that we depict our Young diagrams with row 1 (the longest row) at the bottom; this accords with the literature on plactic monoids, but differs from the literature on representation theory.}. We number the rows from the bottom up. A \emph{(Young) tableau over $[n]$} is any filling of the boxes of a Young diagram with numbers from $[n]$ (one per box)  in such a way that the entries strictly decrease when reading down each column, and weakly increase when reading from left to right along each row\footnote{Such tableaux are often referred to as \textit{semi-standard} in the literature. Since all tableaux considered in this article are semi-standard, we prefer the simpler term `tableaux' throughout.}. The \emph{column-reading} of a tableau is the word 
over $[n]$ obtained by reading the tableau down each column in turn, with the columns ordered left-to-right. Dually, the \emph{row-reading} of a tableau is the word obtained by reading the tableau along each row in turn from left-to-right, with the rows ordered from top-to-bottom.  For example, the tableau
$$\begin{ytableau}
 5 \\
 4 & 5 & 5 \\
 2 & 2 & 3 \\
  1 & 1 & 1 & 2 & 3
 \end{ytableau}
 $$
has column reading $542152153123$ and row reading $545522311123$.

It is well-known that the elements of $\plac{n}$ are in bijective correspondence with the set of tableaux over $[n]$, via the map which identifies a tableau with the equivalence class of words containing its column reading (or equivalently, its row reading). We shall often identify elements of the plactic monoid with the corresponding tableaux. Multiplication within $\plac{n}$ can be understood combinatorially as an application of \textit{Schensted's insertion algorithm} \cite{Schensted}.

The \textit{tropical semiring} $\trop$ is the set $\mathbb{R} \cup\{-\infty\}$ under the operations $a \oplus b = {\rm max}(a,b)$ and $a \otimes b = a+b$, where we define $-\infty + a=a+ -\infty = -\infty$. We write $M_n(\trop)$ to denote the semigroup (which is in fact  a monoid) of all $n \times n$ matrices with entries from $\trop$ under the matrix multiplication induced from operations of $\mathbb{T}$ in the obvious way. We say that $A \in M_n(\trop)$ is \emph{upper triangular} if $A_{i,j}=-\infty$ for all $i>j$, and write $UT_n(\trop)$ for the set (which forms a submonoid) of $n \times n$ upper triangular matrices
over $\trop$. If $S$ is a finite set we write $M_S(\trop)$ for the semigroup of matrices with rows and columns indexed by elements
of $S$; clearly $M_S(\trop)$ is isomorphic to $M_{|S|}(\trop)$ but the indexing by $S$ will sometimes prove helpful.

\section{Tropical Representations of Plactic Monoids}\label{sec_rep}

In this section we shall construct for each finite rank plactic monoid a map to an upper triangular matrix monoid. We then study the properties of
this map, eventually showing that it is in fact a faithful representation. Throughout this section we fix a positive integer $n$. 

For $S \in \twon$ we write $S^i$ (for $i = 1, \dots, |S|$) to denote the $i$th smallest element of $S$ (so $S = \lbrace S^1 < S^2 < \dots < S^{|S|} \rbrace$). If $x \in S$ with $x=S^i$, then we shall say that $i$ is the \textit{row number} of $x$ in $S$, or $x$ is \textit{contained in row $i$ of $S$}.
 We partially order the sets in $\twon$ by $S \leq T$ if $|S| \geq |T|$ and $S^i \leq T^i$ for each $i \in [|T|]$. 
Notice that for any $k \in[n]$, the sets $[k]$ and $\bark_n$ are respectively minimal and maximal among all sets of cardinality $k$. It may assist
the reader to think of sets in $\twon$ as possible columns of tableaux over the alphabet $[n]$; from this viewpoint the ``rows'' of subsets are what one would think, and the partial order is given by $S \leq T$ if the column $T$ can appear to the right of the column $S$ in a tableau.

If $P,Q \in \twon$ we use the notation $[P,Q]$ for the order interval from $P$ to $Q$ (a subset of $\twon$), and $\cup [P,Q]$ for the union of
the sets in the order interval (an element of $\twon$). We say that a word $w \in [n]^*$ is \textit{readable from $P$ to $Q$} if there exists an ordered sequence of sets
$$P \leq S_1 \leq S_2 \leq \dots \leq S_{|w|} \leq Q$$
such that $w_i \in S_i$ for each $i$. We emphasise that there is no requirement that $P = S_1$ or $S_{|w|} = Q$: the sequence of sets $S_i$ merely
needs to be contained within the interval $[P,Q]$. Note also that if $w$ can be read from $P$ to $Q$ then all scattered subwords of $w$ can
be read from $P$ to $Q$.

We define a map $\rho_n : [n]^* \to \utst$ by letting
$$\rho_n(x)_{P,Q} \ = \ 
\begin{cases}
-\infty &\textrm{ if } |P| \neq |Q| \textrm{ or } P \nleq Q; \\
1 &\textrm{ if } |P| = |Q|\textrm{ and } x \in \cup [P,Q]; \\
0 &\textrm{ otherwise (that is, if $|P| = |Q|$, $P \leq Q$, $x \notin \cup [P,Q]$)},
\end{cases}$$
for each generator $x \in [n]$, extending multiplicatively to longer words, and defining the map on the empty word by
$$\rho_n(\epsilon) \ = \ 
\begin{cases}
-\infty &\textrm{ if } |P| \neq |Q| \textrm{ or } P \nleq Q; \\
0 &\textrm{ otherwise.}
\end{cases}$$
(See Remark~\ref{remark_singletonblockcain} below for comments on the reason for choosing $\rho_n(\epsilon)$ to be this
matrix, in preference to the identity matrix in $UT_n(\trop)$.)
It is straightforward to verify that $\rho_n(\epsilon)$ is an idempotent which acts as an identity element for matrices of the form $\rho_n(x)$ with
$x \in [n]$, which is all that is required to show that $\rho_n : [n]^* \to \utst$ is a semigroup morphism. 

Notice that, because the relation $\leq$ is transitive, the position of the $-\infty$ entries in the generators is preserved under multiplication, so
for any $w \in [n]^*$ we have $\rho_n(w)_{P,Q} = -\infty$ if and only if  $|P| = |Q|$ or $P \nleq Q$.

Our main aim in this section is to show that the map $\rho_n$ induces a well-defined, faithful representation of the plactic
monoid $\plac{n}$.
\begin{figure}
$$\left( \begin{array}{c|cccc|cccccc|cccc|c}
1 &&&&&&&&&&&&&&& \\
\hline
& 1 & 1 & 1 & 1 &&&&&&&&&&& \\
&    & 0 & 1 & 1 &&&&&&&&&&& \\
&    &   &  1 & 1 &&&&&&&&&&& \\
&    &   &     & 1 &&&&&&&&&&& \\
\hline
&&&&& 0 & 1 & 1 & 1 & 1 & 1 &&&&& \\
&&&&&    & 1 & 1 & 1 & 1 & 1 &&&&& \\
&&&&&    &    & 1 &    & 1 & 1 &&&&& \\
&&&&&    &    &    & 0 & 0 & 1 &&&&& \\
&&&&&    &    &    &    & 0 & 1 &&&&& \\
&&&&&    &    &    &    &    & 0 &&&&& \\
\hline
&&&&&&&&&&& 0 & 0 & 1 & 1 & \\
&&&&&&&&&&&    & 0 & 1 & 1 & \\
&&&&&&&&&&&    &    & 1 & 1 & \\
&&&&&&&&&&&    &   &    & 0 & \\
\hline
&&&&&&&&&&&&&&& 0
\end{array}\right)$$
\caption{The image of the generator $3$ under the representation $\rho_4 : \plac{4} \to M_{2^{[4]}}(\trop)$. Empty positions represent $-\infty$. See Example~\ref{example_matrix} for the ordering of rows and columns.}
\label{fig_4x4}
\end{figure}

\begin{example}\label{example_matrix}
Figure~\ref{fig_4x4} shows the $16 \times 16$ matrix $\rho_4(3)$, which will be 
the image of the generator $3$ in our representation of $\plac{4}$. Blank entries represent $-\infty$.
The matrix is upper triangular because $\rho_n(x)_{P,Q} = -\infty$ whenever $P \nleq Q$, and we have chosen to order the columns and rows in a way compatible with the order $\leq$. It is block diagonal, because $\rho_n(x)_{P,Q} = -\infty$ whenever $|P| \neq |Q|$ and we have we have grouped together columns and rows corresponding to sets of the same size. It has (from top to bottom):
\begin{itemize} 
\item a $1\times 1$ block indexed by the single set $\{1,2,3,4\}$;
\item a $4\times 4$ block indexed by the subsets $\{1,2,3\} < \{1,2,4\} < \{1,3,4\} < \{2,3,4\}$;
\item a $6\times 6$ block indexed by the subsets $\{1,2\}$, $\{1,3\}$, $\{2,3\}$, $\{1,4\}$, $\{2,4\}$ and $\{3,4\}$;
\item a $4\times 4$ block indexed by the subsets $\{1\} < \{2\} < \{3\} < \{4\}$; and
\item a $1\times 1$ block indexed by the empty set.
\end{itemize}
The two $1 \times 1$ blocks are unimportant, and are included in our definition only because their presence allows a more uniform treatment which simplifies
the proofs; see Remark~\ref{remark_dimension} below. The two
$4 \times 4$ blocks correspond to the representations previously constructed by Izhakian \cite{Izhakian19} and by Cain \textit{et al} \cite{Cain17}; see Remarks~\ref{remark_singletonblockcain}, \ref{remark_izhakiancain} and \ref{remark_bigblockcain}.

The most interesting block is the $6 \times 6$ block corresponding to sets of size two. These sets are not totally ordered because $\{2, 3\}$ and $\{1,4\}$ are incomparable. We have arbitrarily chosen to draw the matrix with $\{1, 4 \}$ before $\{2,3\}$, but the other choice would also have yielded an upper triangular matrix. Note how the fact that $\{ 1, 4 \} \nleq \{ 2, 3 \}$ manifests itself in a $-\infty$ entry above the diagonal in this block. The case $n=4$ is the smallest rank where this behaviour emerges, and this qualitative difference between ranks $3$ and $4$ explains in some philosophical sense why previous methods \cite{Izhakian19,Cain17} which were successful for $\plac{3}$ are not able to faithfully represent $\plac{4}$ (see also Remark~\ref{remark_cainnotfaithful} below).
In yet higher rank there will be many more incomparable subsets, and hence many more $-\infty$ entries above the diagonal within the blocks. 
\end{example}

The following lemma gives a useful combinatorial description of the non-$(-\infty)$ entries of $\rho_n(w)$ for a general word $w$.
\begin{lemma}\label{lemma_scatter}
For every $w \in [n]^*$ and $P,Q \in \twon$ with $|P| = |Q|$ and $P \leq Q$, the entry $\rho_n(w)_{P,Q}$ is the maximum length of a scattered subword
of $w$ that can be read from $P$ to $Q$.
\end{lemma}
\begin{proof}
If $w = \epsilon$ or $|w| = 1$ then the claim is an immediate consequence of the definition of $\rho_n$, so assume $|w| \geq 2$.

Let $i_1 < \dots < i_k$ be a scattered subword of $w$ of maximal length readable from $P$ to $Q$ (chosen arbitrarily if there are multiple such). Let $I = \lbrace i_1, \dots, i_k \rbrace$. 

First we show that $\rho_n(w)_{P,Q} \geq |I|$. If $|I| = 0$ then since $P \leq Q$ we have $\rho_n(w)_{P,Q} \neq -\infty$
so the inequality is immediate. Otherwise, by the definition of readability there exists an ordered sequence of sets
$$P \leq S_{i_1} \leq S_{i_2} \leq \dots \leq S_{i_k} \leq Q$$
such that $w_i \in S_i$ for each $i \in I$. For each $j \in [|w|] \setminus I$ define $S_j$ to be $S_i$ where $i$ is the greatest element of $I$ which is less than $j$,
or $S_i = S_{i_1}$ if $j$ is less than all elements of $I$. For notational convenience, define $S_{|w|+1} = S_{i_k}$. Note that we
have $S_j \leq S_{j+1}$ for all $j \in [|w|]$.

It follows from the definition of $\rho_n$, the fact that $w_i \in S_i$ for all $i \in I$ and the known order relations between the sets $S_i$ that for each $j \in [|w|]$ we have:
\begin{itemize}
\item $\rho_n(w_j)_{S_j,S_{j+1}} = 1$ if $j \in I$; and
\item $\rho_n(w_j)_{S_j,S_{j+1}} \geq 0$ if $j \notin I$.
\end{itemize}
Noting that for each generator $x$, if $P \leq B \leq C \leq Q$ then $ \rho_n(x)_{P,Q} \geq  \rho_n(x)_{B,C}$, we also have
\begin{itemize}
\item $\rho_n(w_1)_{P, S_2} \geq \rho_n(w_1)_{S_1, S_2}$; and
\item $\rho_n(w_{|w|})_{S_{|w|}, Q} \geq \rho_n(w_1)_{S_{|w|}, S_{|w|+1}}$.
\end{itemize}
Now combining these facts with the definition of matrix multiplication,
\begin{align*}
\rho_n(w)_{P,Q} \ &\geq \ \rho_n(w_1)_{P,S_2} + \sum_{j=2}^{|w|-1} \rho_n(w_j)_{S_j,S_{j+1}} + \rho_n(w_{|w|})_{S_{|w|}, Q} \\
&\geq \sum_{j=1}^{|w|} \rho_n(w_j)_{S_j,S_{j+1}} \\
&\geq \ |I|.
\end{align*}

Conversely suppose $\rho_n(w)_{P,Q} = l$. By the definition of matrix multiplication, there exist sets $P = S_1, \dots, S_{|w|+1} = Q$ such
that
$$l \ = \ \rho_n(w)_{P,Q} \ = \ \sum_{j=1}^{|w|} \rho_n(w_j)_{S_j,S_{j+1}}.$$
Since the entries in $\rho_n(x)$ for any
generator $x \in [n]$ are all $0$, $1$ or $-\infty$, this means that none of the terms 
$\rho_n(w_j)_{S_j,S_{j+1}}$ are $-\infty$, and that exactly $l$ of them are $1$. From the former
we deduce that $S_j \leq S_{j+1}$ and that $|S_j| = |P|$ for each $j$. From the latter
we deduce that there is a subset $I \subseteq [|w|]$ such that $|I| = l$ and for all $i \in I$ we have
 $w_i \in \cup [S_i,S_{i+1}]$. For each $i \in I$ we may choose a set $T_i$ such that $S_i \leq T_i \leq S_{i+1}$
(hence $|T_i| = |P|$ too) and $w_i \in T_i$. But now the $T_i$s for $i \in I$ form an ascending sequence in the partial order in
the interval $[P,Q]$, from which it follows that the scattered subword consider of the elements of $I$ in order can be read from $P$ to $Q$. Thus, $w$ has a scattered subword
of length $l$ which can be read from $P$ to $Q$.
\end{proof}

\begin{lemma}\label{lemma_wrongorder}
Suppose that $a, b \in [n]$ with $a < b$ and that the word $ba$ can be read from $P$ to $Q$, where $P, Q \in \twon$ with $|P|=|Q|$. Then there exists a set $S$ in the interval $[P,Q]$ such that $a, b \in S$. In particular, the word $ab$ can also be read from $P$ to $Q$.
\end{lemma}
\begin{proof}
Let $k = |P| = |Q|$.
By definition, there exists an ordered sequence of sets $P \leq S \leq U \leq Q$  such that $b \in S$ and $a \in U$. By replacing $S$ if necessary we may clearly assume that $S$ is maximal with this property, that is, there is no set $S'$ with $S < S' \leq U$ and $b \in S'$. We will show that for this
choice of $S$ one has $a,b \in S$. If $S = U$ then we immediately have $a, b \in S$ and we are done. Suppose then that $S<U$.

Let $i, j \in [k]$ be indices such that $S^j = b$ and $U^i = a$. Since $S < U$ and $b > a$ we must have $i < j$.
Notice that if there was a position $q$ with $i < q < j$ and $S^q = U^q$ then 
setting
$$T \ = \ \{U^1< \cdots< U^q = S^q < S^{q+1}< \cdots< S^k\}$$
would give $S \leq T \leq U$ and $a, b \in T$, which by the maximality of $S$ means $S=T$ and $a,b \in S$ as required.

Assume then that there is no such position, that is, that $S^q \neq U^q$ for all $i < q < j$. In this case we claim that for every $1 \leq m < j-i$, we have $S^{j-m} = b-m$. Indeed, if not then there is some $1 \leq m < j-i$ such that $S^{j-m} < S^{j-m+1}-1$.
Replacing $S^{j-m}$ with $S^{j-m}+1$ then yields a set $T$ with $b \in T$ and $S < T \leq U$, where the latter inequality holds because we have assumed that $S^{j-m} \neq U^{j-m}$. This contradicts the maximality assumption on $S$, establishing the claim.

Now if there exists $1 \leq m < j-i$ with $S^{j-m} = a$ then $a,b \in S$ and we are done. Otherwise we have $S^{i+1} > a$, in which case setting
$$T \ = \ \{U^1< \cdots< U^{i}< S^{i+1}< \cdots<  S^j < \cdots < S^k\}$$
once again gives $S \leq T \leq U$ and $a,b \in T$, and so by maximality of $S$ we must have $S=T$ and $a,b \in S$.
\end{proof}

\begin{lemma}\label{lemma_morphism}
The map $\rho_n$ induces a well-defined morphism from $\plac{n}$ to $\utst$.
\end{lemma}
\begin{proof}
It suffices to show that $\rho_n$ respects the Knuth relations which define $\plac{n}$, that is, for each defining relation $u=v$ we have
$\rho_n(u) = \rho_n(v)$.

The defining relations each have both sides of length $3$, and in all cases both sides of the relation feature the same letters with the
same multiplicity, with only the order differing. Notice that for any word $u$ of length $3$, and any $P,Q \in \twon$ we have the following
facts, which follow from the block diagonal structure of our matrices together with Lemma \ref{lemma_scatter}:
\begin{itemize}
\item Every entry of  $\rho_n(u)$ is $-\infty$, $0$, $1$, $2$ or $3$ (since the maximum contribution given by each letter is $1$).
\item $\rho_n(u)_{P,Q} = -\infty$ if and only $P \not\leq Q$ or $|P| \neq |Q|$.
\item $\rho_n(u)_{P,Q} = 0$ if and only if the support of $u$ does not intersect with  $\cup [P,Q]$.
\item $\rho_n(u)_{P,Q} = 1$ if and only $u$ contains a single (unrepeated) letter from $\cup [P,Q]$.
\end{itemize}
It follows that if $u=v$ is a defining relation then $\rho_n(u)_{P,Q} = -\infty$ [respectively $0$, $1$] if and only if $\rho_n(v)_{P,Q} = -\infty$ [respectively $0$, $1$], so it will suffice to show that $\rho_n(u)_{P,Q} = 3$ if and only if $\rho_n(v)_{P,Q} = 3$. By Lemma~\ref{lemma_scatter}, this means it will
suffice to show that $u$ can be read from $P$ to $Q$ if and only if $v$ can be read from $P$ to $Q$, for each defining relation $u=v$ and each $P,Q \in \twon$ with $|P| = |Q|$. We do this by analysing separately the two types of relation. Suppose, then, that $|P| = |Q| = k$.

\textbf{Case 1}. Suppose $u=bca$ and $v=bac$ where $a< b \leq c$.
If $u$ can be read from $P$ to $Q$, then there is a sequence $P \leq S_1 \leq S_2 \leq S_3 \leq Q$ with $b \in S_1, c\in S_2$ and $a \in S_3$. By Lemma~\ref{lemma_wrongorder} (applied to the descending word $ca$ which can be read from $S_2$ to $S_3$) there exists $T$ with $P \leq S_1 \leq S_2 \leq T \leq S_3 \leq Q$ and $a,c \in T$, whence $v$ can also be read from $P$ to $Q$.

Conversely, suppose that $v$ can be read from $P$ to $Q$, that is, there is a sequence $P \leq S_1 \leq S_2 \leq S_3 \leq Q$ with $b \in S_1, a \in S_2$ and $c \in S_3$. By Lemma~\ref{lemma_wrongorder} (applied to the descending word $ba$ which can be read from $S_1$ to $S_2$) there exists $T$ with $P \leq S_1 \leq T \leq S_2 \leq S_3 \leq Q$ and $a,b \in T$. Notice that if $c \in T$ (in particular, if $b=c$), then we are done, since in this case $u$ can also be read from $P$ to $Q$. Otherwise let $U =S_3$ and let $j, i, m \in [k]$ be the indices such that $T^j =b$, $T^i=a$ and $U^m=c$. Notice that $i < j$. There are two configurations to consider:

\begin{itemize}
\item \textbf{Case 1A.} Suppose $i<m$. In this case set
$$R \ =\ \{ T^1 < \cdots < T^{m-1} < U^{m} < \cdots < U^k \}.$$
The fact that $T^{m-1} < U^m$ follows
from the facts that $T^{m-1} < T^m$ and $T \leq U$. Since $T \leq U$ it is clear that $T^p \leq R^p \leq U^p$ for all $p$,
so $T \leq R \leq U$.
\item \textbf{Case 1B.} Suppose $i\geq m$. In this case set
$$R \ = \ \{T^1 < \cdots < T^{j-1} <  c < U^{j+1} < \cdots < U^k\}.$$
The fact that $T^{j-1} < c$ is because $T^{j-1} < T^j = b < c$, while the fact $c < U^{j+1}$ is because $c = U^m \leq U^i < U^j < U^{j+1}$.  Since $T \leq U$ it is immediate that $T^p \leq R^p \leq U^p$ for all $p \neq j$. Noting that $T^j=b<c=U^m<U^j$ we again conclude that $T \leq R \leq U$.
\end{itemize}
Thus in each case we have $P \leq T \leq R \leq U \leq Q$ with $b \in T$ and $a, c \in R$ (since $a = T^i$ and $c = U^m$), showing that both $u$ and $v$ can be read from $P$ to $Q$.

\textbf{Case 2.} Suppose now that $u=cab$ and $v=acb$ where $a\leq b < c$.
This case is treated similarly to Case 1. If $u$ can be read from $P$ to $Q$, then there is a sequence $P \leq S_1 \leq S_2 \leq S_3 \leq Q$ with $c \in S_1, a\in S_2$ and $b \in S_3$. By Lemma~\ref{lemma_wrongorder} (applied to the descending word $ca$ which can be read from $S_1$ to $S_2$) there exists $T$ with $P \leq S_1 \leq T \leq S_2 \leq S_3 \leq Q$ and $a,c \in T$, whence $v$ can also be read from $P$ to $Q$.

Conversely, suppose that $v$ can be read from $P$ to $Q$, that is, there is a sequence $P \leq S_1 \leq S_2 \leq S_3 \leq Q$ with $a \in S_1, c \in S_2$ and $b \in S_3$. By Lemma~\ref{lemma_wrongorder} (applied to the descending word $cb$ which can be read from $S_2$ to $S_3$) there exists $T$ with $P \leq S_1  \leq S_2 \leq T \leq  S_3 \leq Q$ and $c,b \in T$. Notice that if $a \in T$ (in particular, if $a=b$), then we are done, since in this case $u$ can also be read from $P$ to $Q$. Otherwise let $S = S_1$ and let $j, i, m \in [k]$ be indices such that $S^m =a$,  $T^i=b$ and $T^j=c$. Then $i < j$ and there are again two configurations to consider:

\begin{itemize}
\item \textbf{Case 2A.} Suppose $m<j$. In this case set
$$R \ =\ \{ S^1 < \cdots < S^{j-1} < T^{j} < \cdots < T^k \}.$$
The fact that $S^{j-1} < T^j$ follows
from the facts that $S^{j-1} < S^j$ and $S \leq T$. Since $S \leq T$ it is clear that $S^p \leq R^p \leq T^p$ for all $p$, so $S \leq R \leq T$.
\item \textbf{Case 2B.} Suppose $m \geq j$. In this case set
$$R \ = \ \{S^1 < \cdots < S^{i-1} < a < T^{i+1} < \cdots < T^k \}.$$
The fact that $S^{i-1} < a$ is because $S^{i-1} < S^j \leq S^m = a$, while the fact $a < T^{i+1}$ is because $a < b = T^i < T^{i+1}$.  Since $S \leq T$ it is immediate that $S^p \leq R^p \leq T^p$ for all $p \neq i$. Since $S^i<S^m=a<b=T^i$, we again have $S \leq R \leq T$.
\end{itemize}
In each case we have
$P \leq S \leq R \leq T \leq Q$ with $a, c \in R$ (since $a = S^m$ and $c = T^j$) and $b \in T$, showing that both $u$ and $v$ can be read from $P$ to $Q$.
\end{proof}

We shall, in a slight abuse of notation, also denote by $\rho_n$ the induced map from $\plac{n}$ to $\utst$. In particular we sometimes write
$\rho_n(T)$ where $T$ is a tableau over $[n]$.

\begin{lemma}\label{lemma_nodescending}
If $w \in [n]^*$ contains a strictly descending scattered subword of length $k+1$ or more, then $w$ cannot be read from $[k]$ to $\bark_n$.
\end{lemma}
\begin{proof}
Suppose for a contradiction that $w$ can be read from $[k]$ to $\bark_n$ and contains a strictly descending scattered subword of length $k+1$ or more. By replacing $w$ with a scattered subword, we may assume that $w$ is actually equal to a strictly descending word of length $k+1$ which can be read
from $[k]$ to $\bark_n$. Then by definition there exists an ordered sequence of sets
$$[k] \leq S_1 \leq S_2 \leq \dots \leq S_{|w|} \leq \bark_n$$
such that $w_i \in S_i$ for each $i$. For each $i$, let $r_i \in [k]$ be the row number of $w_i$ in $S_i$. The fact that
$S_i \leq S_{i+1}$ for $1 \leq i < |w|$ means every letter in row $r_i$ or above in $S_{i+1}$ is greater than or equal to $w_i$. Since
$w$ is strictly descending we must therefore have $r_{i+1} < r_i$. But now $r_1, r_2, \dots, r_{k+1}$ is a strictly descending sequence of length
$k+1$ in the set $[k]$, giving the required contradiction.
\end{proof}

To prove the following lemma we shall need a new definition.
Let $k \leq m \leq n$. If $S\subseteq [m]$ is a set of cardinality $k$ or less we define the \textit{$(k,m)$-completion of $S$} to be the set of cardinality $k$ obtained by adding
in the $k-|S|$ largest values from $[m] \setminus S$. Notice that if $\hat{S}$ is the $(k,m)$-completion of $S$ then we have $\hat{S} \subseteq [m]$ and for all $i \in [k]$, at least one of the following holds:
\begin{itemize}
\item $\hat{S}^i=m-k+i \leq S_i$; or
\item $i \leq |S|$ and $\hat{S^i} = S^i \leq m-k+i$.
\end{itemize}
In particular $S^i \geq \hat{S}^i$ for all $i \in [|S|]$.
(To immediately see these facts it may be helpful to think about the sets concerned as tableau columns, as described above.)

\begin{lemma}\label{lemma_kcompletion}
Suppose $S, T \in 2^{[m]}$ and let $\hat{S}$ and $\hat{T}$ be the $(k,m)$-completions of $S$ and $T$ respectively.
If $S \leq T$ then $\hat{S} \leq \hat{T}$.
\end{lemma}
\begin{proof}
This follows from the facts immediately above the statement of the lemma, along with the definition of the order on sets. 
Indeed, consider the entries of the various sets in row $i$ for $i \in [k]$. If $\hat{T}^i= m-k+i$, then in particular it is greater than or equal to $\hat{S}^i$. Otherwise we have $i \leq |T| \leq |S|$ (so that $T^i$ and $S^i$ are defined)
and
$$\hat{T}^i \ = \ T^i \ \geq \ S^i \ \geq \ \hat{S}^i.$$
So in all cases $\hat{T}^i \geq \hat{S}^i$, which since $|\hat{S}| = |\hat{T}|$ means that $\hat{S} \leq \hat{T}$.
\end{proof}

\begin{lemma}\label{lemma_rowcount}
Let $T \in \plac{n}$ be a tableau and $k,m \in [n]$ with $k \leq m$. Then $\rho_n(T)_{[k],[\hat{k}]_m}$ is the total number of entries from $[m]$ lying in the bottom $k$ rows of $T$.
\end{lemma}
\begin{proof}
Let $w \in [n]^*$ be the column reading of $T$. Then by Lemma \ref{lemma_morphism} $\rho_n(T) = \rho_n(w)$, and by
Lemma~\ref{lemma_scatter}, $\rho_n(T)_{[k],[\hat{k}]_m} = \rho_n(w)_{[k],[\hat{k}]_m}$ is the maximum length of a scattered subword of $w$ that can be read from $[k]$ to $[\hat{k}]_m$. Notice that if $[k] \leq P \leq [\hat{k}]_m$, then $P \subseteq [m]$. Thus it suffices to consider scattered subwords of $w$ which use only the letters from $[m]$.

First, let $v$ be the scattered subword of $w$ consisting of those letters in $[m]$ coming from the bottom $k$ rows of the tableau. For each $i$ let $S_i$ be the set of letters from $[m]$ in the bottom $k$ rows of the column of the tableau from which the letter $v_i$ originates. The $v$ comes from the column reading of $T$, each $S_i$ is either equal to $S_{i+1}$ or appears to the left of $S_{i+1}$ in some tableau, so we have $S_i \leq S_{i+1}$ for
all $i$. Let $T_i$ be the
$(k,m)$-completion of $S_i$. Then by Lemma~\ref{lemma_kcompletion}, we have $T_i \leq T_{i+1}$ for all $i$.
Since every subset of $[m]$ of cardinality $k$ is above $[k]$ and below $[\hat{k}]_m$ we have
$$[k] \leq T_1 \leq \dots \leq T_{|v|} \leq [\hat{k}]_m$$
and $v_i \in S_i \subseteq T_i$ for each $i$, which shows that $v$ can be read from $[k]$ to $[\hat{k}]_m$. Thus $\rho_n(s)_{[k],[\hat{k}]_m}$ is greater than or equal to the total number of entries from $[m]$ lying in the bottom $k$ rows of the tableau representation of $s$.

Next note that any scattered subword of $w$ using only the letters of $m$ that is strictly longer than $v$ must clearly contain $k+1$ or more letters from some column in the tableau, and hence must contain a strictly descending scattered subword of length $k+1$. But by Lemma~\ref{lemma_nodescending} this means this scattered subword cannot be read from $[k]$ to $\bark_n$. Since $[k] \leq \bark_m \leq \bark_n$
this means it cannot be read from $[k]$ to $[\hat{k}]_m$, giving a contradiction.

We have shown that $v$ is a scattered subword of maximal length that can be read from $[k]$ to $[\hat{k}]_m$, which establishes the claim.
\end{proof}

\begin{theorem}\label{thm_main}
The map $\rho_n : \plac{n} \to \utst$ is a faithful representation of $\plac{n}$.
\end{theorem}

\begin{proof}
The fact that the given map is a well-defined morphism is Lemma~\ref{lemma_morphism}. To see that it is faithful, we observe that any tableau $T \in \plac{n}$ can be reconstructed its image under this morphism. It will suffice to recover for each $m \in [n]$ and $k \in [n]$, the number of occurrences of the symbol $m$ in row $k$ of the tableau. By a simple inclusion-exclusion argument, this number is the number of entries from $[m]$ in the bottom $k$ rows, minus the number of entries from $[m]$ in the bottom $k-1$ rows, minus the number of entries from $[m-1]$ in the bottom $k$ rows, plus the number of entries from $[m-1]$ in the bottom $k-1$ rows. Each of these four values is either immediately seen to be $0$, or else by Lemma~\ref{lemma_rowcount} can be seen in the matrix $\rho_n(T)$.
\end{proof}

\begin{remark}\label{remark_singletonblockcain} The block of our representation $\rho_n$ corresponding to singleton sets is essentially the same as one of the blocks in the (non-faithful, for $n \geq 4$) representation given by Cain \textit{et al} \cite{Cain17}. 
Indeed, if $P = \{p\}$ and $Q =\{q\}$ are singleton sets with $p \leq q$, and $w \in [n]^+$ then it follows easily from Lemma~\ref{lemma_scatter} that
$\rho_n(w)_{P,Q}$ is the maximum length of a non-decreasing scattered subword of $w$ whose entries lie in the range $[p,q]$, which is (by definition)
the value of $\phi_n(w)$ in \cite{Cain17}. We shall use this description of $\rho_n(w)_{P,Q}$ in Section~\ref{sec_plactictotropical} below. The
only difference between these two representations (apart from the formal difference in row and column labels) is that $\phi_n(\epsilon)$ is defined in \cite{Cain17} to be the identity matrix of $UT_n(\trop)$. That approach has the advantage of yielding a morphism of monoids, but at the price of what in our view is a more artificial definition which more often necessitates treating the
empty word as a special case. (For example, Lemma~\ref{lemma_scatter} above would not hold as stated if the empty word were mapped to the identity matrix.)
\end{remark}

\begin{remark}\label{remark_izhakiancain}
It can be seen from the proof of \cite[Lemma 7.6]{Izhakian19} that the representation $\mathscr{C}_{\rm mat}$ defined by Izhakian also acts on the generators in the same way as the singleton set block of our $\rho_n$ and the map $\phi_n$ of \cite{Cain17}, and thus also defines essentially the same representation. (The identity element is sent to the same element as in our representation, rather than to the identity element
of $UT_n(\trop)$ as in \cite{Cain17}.) 
\end{remark}

\begin{remark}\label{remark_bigblockcain}
Similarly, the block of our representation corresponding to sets of cardinality $n-1$ is closely related to the representation $\sigma_n$ in \cite{Cain17}, and is also closely related to the other block of the representation of \cite{Izhakian19}. We do not have such a simple conceptual explanation for this, but straightforward calculations show that each generator of $\plac{n}$ is sent to the same matrix by the $(n-1)$-block of our representation as by the representation $\sigma_n$, so the two representations differ only in where they send the identity element. Specifically, generator $i$ is sent in both cases to the matrix with $0$ in the $(n+1-i)$th diagonal position
and $1$ in all other positions on or above the diagonal. The representation $\mathscr{C}_{\rm mat}^{\rm co}$ of Izhakian \cite{Izhakian19} sends $i$ to the matrix with $-1$ in the $(n+1-i)$th diagonal position and $0$ in all other positions on or above the diagonal, that is, the same matrix tropically scaled by $-1$. Thus for every $w \in [n]^+$ we have $\mathscr{C}_{\rm mat}^{\rm co}(w) = (-|w|) \otimes \sigma_n(w)$. Since tropical scaling is injective, and since two words of different lengths never represent the same element of the plactic monoid, it follows that the representation $\mathscr{C}_{\rm mat}^{\rm co}$ carries the same information as the other two (and hence that the full representations constructed in \cite{Izhakian19} and \cite{Cain17} carry exactly the same information as each other).
\end{remark}

\begin{remark}\label{remark_cainnotfaithful}
In rank $4$, a simple calculation shows that the blocks of our representation corresponding to sets of size $1$ and $3$ do not distinguish between,
for example, the elements
$$\ytableausetup{centertableaux}
\begin{ytableau}
 4 & 4 \\
 2 & 3 & 4 \\
 1 & 2 & 3 & 3
\end{ytableau} \ \ \ \ \ \textrm{ and } \ \ \ \ \ \begin{ytableau}
 4 \\
 2 & 3 & 4 & 4\\
 1 & 2 & 3 & 3
\end{ytableau}.
$$
Rather, these are distinguished in the block corresponding to sets of size $2$, specifically in the $(\lbrace 1,2 \rbrace , \lbrace 3,4 \rbrace)$ entry, which by Lemma~\ref{lemma_rowcount} is $7$ for the left-hand tableau and $8$ for the right-hand tableau.
It follows from this and Remarks~\ref{remark_singletonblockcain}, \ref{remark_izhakiancain} and \ref{remark_bigblockcain} that these elements of $\plac{4}$ are not
distinguished by the representations of Izhakian \cite{Izhakian19} or of Cain \textit{et al} \cite{Cain17}. (A similar example is given by Izhakian
to show that his representation is not faithful \cite[example before Theorem 7.18]{Izhakian19}.)
\end{remark}

\begin{remark}\label{remark_dimension}
In the interests of a ``clean'' proof of Theorem~\ref{thm_main}, we have not attempted
to optimise the dimension of our representation for each $\plac{n}$. The dimension of the representation as given is $2^n$,
but it can be shown that certain blocks (in particular, those corresponding to the empty set and the set $[n]$) of the representations are redundant and the dimension can therefore be slightly reduced. However, we believe that any such pruning of our representation will yield something of the order $2^n$,
and we conjecture that it is not possible to improve significantly upon this:
\end{remark}
\begin{conjecture}
Any family of faithful tropical representations for $\plac{n}$ has dimension which grows exponentially, with base at least $2$, as a function of $n$.
\end{conjecture}

\section{From Tropical to Plactic Identities}\label{sec_tropicaltoplactic}

Since it is known that upper triangular tropical matrix semigroups of each rank
satisfy non-trivial semigroup identities \cite{Izhakian14,Okninski15,Taylor17}, Theorem~\ref{thm_main} has the following
immediate corollary, which establishes a conjecture of Kubat and Okni\'{n}ski \cite{Kubat15}:
\begin{theorem}
The plactic monoid of each finite rank satisfies a non-trivial semigroup
identity.
\end{theorem}
A proof of this statement was also proposed in a recent preprint of Okni\'{n}ski \cite{Okninski19}, but the
preprint has been retracted due to an incorrect proof.

We remarked above (Remark~\ref{remark_dimension}) that we have not attempted to optimise the dimension of our representations.
For some purposes (in particular, the study of identities) the dimension of an upper triangular representation is
not the most important parameter: rather, what matters is the ``chain
length'' in the sense of the following definition introduced by Daviaud and the present authors \cite{Daviaud18}.

Let $\Gamma$ be a finite partially ordered set and let $N$ be the least upper bound on the length of ascending chains
in $\Gamma$. Let
$$\Gamma(\mathbb{T}) = \{A \in M_{\Gamma}(\trop) \mid A_{P,Q} \neq -\infty \implies P \leq Q\}.$$
Then $\Gamma(\mathbb{T})$ is a subsemigroup of $M_{\Gamma}(\trop)$, called a 
\emph{chain-structured tropical matrix semigroup}\footnote{In fact this is a simplification of the definition in \cite{Daviaud18},
since the latter allows $\Gamma$ to be infinite, which introduces some additional complications to be considered in the definition; here we shall need only the
case where $\Gamma$ is finite.} of \emph{chain length} $N$.

\begin{theorem}
For each $n$, $\plac{n}$ embeds in a chain-structured
tropical matrix semigroup of chain length the integer
part of $\frac{n^2}{4}+1$.
\end{theorem}
\begin{proof}
Let $\Gamma$ be the partial order with underlying set $\twon$ and $\preceq$
given by $S \preceq T$ if and only if $|S| = |T|$ and $S \leq T$ in the order previously defined.
Since this is exactly the condition for the entries in the generators
of our representation to be different from $-\infty$, the obvious map from our representation
to the chain-structured semigroup $\Gamma(\trop)$ is an embedding.

Clearly the chain length is the length of the longest chain  with respect to the order $\leq$ of the sets
of the same size. If $|S| = |T| = k$ then it is easy to see that if $S \leq T$ then
the sum of the entries in $T$ (viewed simply as integers) strictly
exceeds that in $S$. Moreover, if these sums differ by more than $1$ then one can interpolate
between $S$ and $T$ with an ordered sequence of sets where the sum of
the entries in successive sets differs by $1$.

The minimum and maximum values of this sum are clearly
attained for $[k]$ and $\bark_n$ respectively, and the difference between
these values is $k(n-k)$, so the length of the longest possible chain of
sets of cardinality $k$ is $k(n-k)+1$, the maximum possible
value of which is $\frac{n^2}{4}+1$ or $\frac{(n+1)(n-1)}{4}+1$ depending on whether
$n$ is even or odd. Both of these are the integer part of $\frac{n^2}{4}+1$.
\end{proof}

Daviaud and the present authors showed \cite[Theorem~5.3]{Daviaud18} that chain-structured
tropical matrix semigroups of chain length $n$ satisfy the
same identities as $UT_n(\trop)$. Thus, although our faithful representations of $\plac{n}$ requires
matrices of dimension exponential in $n$, we can deduce from them $\plac{n}$ satisfies identities satisfied
by upper triangular tropical matrices of dimension only quadratic in $n$.
\begin{corollary}\label{cor_identities}
For each $n$, $\plac{n}$ satisfies all semigroup identities satisfied by $UT_d(\trop)$
where $d$ is the integer part of $\frac{n^2}{4}+1$.
\end{corollary}
In the particular case that $n=3$, we recover Izhakian's result \cite[Corollary 7.19]{Izhakian19} that $\plac{3}$ satisfies all identities satisfied by $UT_3(\trop)$; we shall see below (in Corollary~\ref{cor_rank3varieties}) that in this case the converse also holds.

As remarked in the introduction, Cain \textit{et al} \cite[Proposition 3.1]{Cain17} have shown that no single identity is satisfied by $\plac{n}$ for all $n$. Combining this with Corollary~\ref{cor_identities} (or directly with Theorem~\ref{thm_main}) we deduce the following statement which, although perhaps not surprising, has to the best of our knowledge not previously been shown:
\begin{corollary}\label{cor_utgrow}
No single semigroup identity is satisfied by $UT_n(\trop)$ for all finite $n$.
\end{corollary}
We note that Corollary~\ref{cor_utgrow} does not automatically imply that each semigroup $UT_{n+1}(\trop)$ satisfies strictly fewer identities than $UT_{n}(\trop)$, but we expect that this is the case:
\begin{conjecture}\label{conj_successiveut}
For every positive integer $n$ there is a semigroup identity satisfied by $UT_n(\trop)$ but not by $UT_{n+1}(\trop)$.
\end{conjecture}

\section{From Plactic to Tropical Identities}\label{sec_plactictotropical}

In the previous section we proved that the variety generated by the plactic monoid $\plac{n}$ is contained in the variety generated by the monoid $UT_d(\trop)$ of upper triangular tropical matrices, where $d$ is the integer part of $\frac{n^2}{4}+1$. In this section we prove a result in the converse direction: that the variety generated by $\plac{n}$ contains the variety generated by $UT_n(\trop)$.

Our proof strategy is similar in structure to (but technically much more complex than) the argument establishing the main theorem in
\cite{Daviaud18}: we show that any identity which can be falsified in ${\rm UT}_n(\mathbb{T})$ can be falsified by matrices lying in a particular homomorphic image of $\plac{n}$, and hence also in $\plac{n}$ itself. The specific homomorphism we shall use is the $n$-dimensional block of our representation $\rho_n$ corresponding to the singleton subsets of $[n]$ (which, as we saw in Remark~\ref{remark_singletonblockcain}
and \ref{remark_izhakiancain}, is essentially the same as representations constructed in \cite{Izhakian19} and \cite{Cain17}). In this section we will
identify singleton subsets with their elements, and write $\rho : \plac{n} \to UT_n(\trop)$ for this representation. We shall need some (rather technical) sufficient conditions for a matrix to lie in the image of $\rho$.

\begin{remark}\label{remark_tabcond}
Let $T$ be a tableau over $[n]$, and for $1 \leq x \leq y \leq n$ let $i_{x,y}$ denote the number of times the symbol $y$ appears in row $x$ of $T$. Clearly
the tableau $T$ is exactly characterised by the values $i_{x,y}$; we refer to these as the \textit{parameters} of the tableau.
(They are the entries of the \textit{configuration tableau} associated to a semi-standard tableau in \cite[Section 6.3]{Izhakian19}.)
It follows easily from
the tableau properties that
$$\sum_{y=x}^{x+t}  i_{x,y} \ \geq \ \sum_{y=x+1}^{x+t+1}  i_{x+1,y}$$
for all $1 \leq x < n$ and $0 \leq t < n-x$.
Conversely, any family of non-negative integer values $i_{x,y}$ which satisfies these inequalities determines a tableau
$T$ of which the values are the parameters.
\end{remark}

\begin{lemma}
\label{imentries}
Let $T$ be a tableau with parameters $i_{x,y}$ for $1 \leq x \leq y \leq n$, and suppose $i_{x+1,y+1} \geq i_{x, y}$ for all $1 \leq x < y<n$. Then
$$\rho(T)_{p,q} \ = \ \sum_{j=p}^{q-1} i_{p,j} + \sum_{j=1}^{p} i_{j,q} \ \ \textrm{ for all } 1 \leq p \leq q \leq n.$$
\end{lemma}
\begin{proof}
Let $w \in [n]^*$ be the row reading of $T$.
By Remark~\ref{remark_singletonblockcain}, $\rho(T)_{p,q}$ is the maximum length of a non-decreasing scattered subword of
$w$ contained in $[p,q]^*$.
Clearly the right-hand-side of the given equation is the length of one such scattered subword: namely that consisting of all
symbols from $[p,q]$ in row $p$, followed by all occurrences of the symbol $q$ in rows below; call this subsequence $\chi$. It will suffice to show that $\chi$ is of maximal length among non-decreasing scattered subwords of $w$.

Let $S$ be the set of non-decreasing scattered subwords of $w$ contained in $[p,q]^*$ and of maximal length. Let $S'$ be the
subset of $S$ consisting of those scattered subwords that only include letters from row $p$ of $T$ and below. We claim first that $S'$ is non-empty.
Indeed, given a sequence in $S$, it is easy to see (bearing in mind that a non-decreasing subsequence of $w$ can include at most one symbol
read read each column of $T$) that one can obtain another sequence in $S$ of the same length by replacing each letter read from strictly above row $p$ with the letter directly below it in row $p$. 

Now for $\alpha \in S'$ and
$i \in [p,q]$ we let $\alpha^{(i)}$ denote the number of the highest row of $T$ from which $\alpha$ includes a symbol greater than or equal to $i$, or $0$
if $\alpha$ does not contain such a symbol. We partially order $S'$ by $\alpha \leq \beta$ if $\alpha^{(i)} \leq \beta^{(i)}$ for all $i \in [p,q]$.
Since $S'$ is finite and non-empty, this order must have at least one maximal element; we claim that $\chi$ is the unique maximal element.

Indeed, suppose for a contradiction that some other element $\alpha \in S'$, not equal to $\chi$, is maximal. Then since $\alpha$  is not $\chi$, but is at least as long as $\chi$, it must read some symbol $y \neq q$ in some row $x \neq p$. Since $\alpha \in S'$, we must have $x < p$, and since
$\alpha$ is a word in $[p,q]^*$ we must have $y < q$. We may
suppose without loss of generality that $x$ is the lowest row from which $\alpha$ reads a symbol less than $q$, and that $y$ is the least symbol which $\alpha$ reads from row $x$.

By assumption $i_{x+1,y+1} \geq i_{x,y}$, which means there are at least as many occurrences of the symbol $y+1$ in row $x+1$, as there are occurrences
of $y$ in row $x$ (so in particular, there is at least one occurrence of $y+1$ in row $x+1$). Let $\beta$ be the scattered subsequence obtained by modifying 
$\alpha$ to read all of the $(y+1)$s in row $x+1$ instead of
the $y$s read by $\alpha$  in row $x$. Then $\beta$ is non-decreasing (because row $x$ was the lowest row from which $\alpha$ reads $y$, and $y$ was the lowest symbol which $\alpha$ reads from row $x$). Moreover, $\beta$ is 
at least as long as $\alpha$, which since $\alpha$ was assumed to be of maximal length means $\beta$ is also of maximal length. Thus, $\beta \in S'$.

We claim that $\alpha < \beta$. Clearly $\alpha^{(y+1)} \leq x$ (because $\alpha$ reads $y$s from row $x$, and so cannot read anything greater than
$y$ strictly above row $x$) while $\beta^{(y+1)} \geq x+1$ (because $\beta$ reads $(y+1)$s in row $x+1$); hence $\alpha^{(y+1)} < \beta^{(y+1)}$.
Now if $\alpha$ reads $y$s from a row strictly above row $x$ then we have $\alpha^{(y)} = \beta^{(y)}$, while if it doesn't then
$\alpha^{(y)} = x$ and $\beta^{(y)} = x+1$ (since $\beta$ reads $(y+1)s$ in row $x+1$); so in all cases $\alpha^{(y)} \leq \beta^{(y)}$. Finally,
consider an $i \notin \lbrace y, y+1 \rbrace$.
If $i < y$ and no symbols from $[i, y-1]$ appear in $\alpha$ then also no symbols from $[i,y-1]$ appear in $\beta$, so we have 
$\alpha^{(i)} = \alpha^{(y)} \leq \beta^{(y)} = \beta^{(i)}$; otherwise, we have $\alpha^{(i)} = \beta^{(i)}$, so in all cases $\alpha^{(i)} \leq \beta^{(i)}$,
completing the proof of the claim that $\alpha < \beta$.

But this contradicts the maximality assumption on $\alpha$. Hence, we conclude that $\chi$ must be the unique maximal element of $S'$, so in particular $\chi \in S$ as required to complete the proof.
\end{proof}

\begin{lemma}
\label{suffinequal}
Let $X \in UT_n(\trop)$ be an upper triangular tropical matrix with non-negative integer entries on and above the main diagonal, and
satisfying the following inequalities:
\begin{itemize}
\item[(A)] $\sum_{i=1}^n X_{i,i} \geq \sum_{i=1}^{n-1} X_{i,i+1}$;
\item[(B)] $X_{i,j+1} \geq  X_{i,j}$ for all $1 \leq i \leq j < n$;
\item[(C)] $ X_{i-1,j} \geq X_{i,j}$ for all $1 < i \leq  j \leq n$;
\item[(D)] $X_{i,j}+X_{i+1,j+1} \geq X_{i,j+1} + X_{i+1,j}$ for all $1 \leq i <j <n$.
\end{itemize}
Then $X = \rho(T)$ for some $s \in \mathbb{T}_n$. 
\end{lemma}
(We remark that properties (B) and (C) together amount to the matrices being \textit{synoptic} in the sense defined in \cite{Izhakian19}).)
\begin{proof}
For $0 \leq x \leq y \leq n$ let $f(x,y) =\sum_{i=1}^x X_{i,y-x+i}$. We record the following useful fact:
\begin{equation}
\label{fact}
f(s+1,t+1)-f(s,t)=X_{s+1,t+1}\,  \mbox{ for all } \,0 \leq s \leq t <n,
\end{equation}
which is immediate from this definition.

Now for $1 \leq x \leq y \leq n$ define
$$i_{x,y} = \begin{cases} f(x,y)-f(x-1,y)-f(x,y-1)+f(x-1,y-1)& \mbox{ if } x <y\\
f(x,x)-f(x-1,x) & \mbox{ if } x =y,
\end{cases}$$
We shall first use the conditions in Remark~\ref{remark_tabcond} to show that there is a tableau $T$ with these parameters, and then show that $T$ satisfies the conditions of Lemma~\ref{imentries}, so that $\rho(T)_{p,q}=X_{p,q}$ for all $p$ and $q$. 

Clearly since the entries of $X$ are assumed to be integers, the parameters $i_{x,y}$ are integers. We claim that they are non-negative. Indeed, when $y=x$ one finds that
\begin{eqnarray*}
i_{x,y} &=& f(x,x)-f(x-1,x) \\
& =& \left( \sum_{i=1}^n  X_{i,i} -  \sum_{i=x+1}^n X_{i,i} \right) - \left( \sum_{i=1}^{n-1} X_{i,i+1} - \sum_{i=x}^{n-1} X_{i,i+1} \right) \ \\
& =& \left( \sum_{i=1}^n  X_{i,i} -  \sum_{i=1}^{n-1} X_{i,i+1} \right) - \left( \sum_{i=x+1}^n X_{i,i} - \sum_{i=x}^{n-1} X_{i,i+1} \right) \ \\
& \geq & \sum_{i=x}^{n-1} X_{i,i+1} - \sum_{i=x+1}^n X_{i,i} \qquad \mbox{ by (A)}\\
& \geq & \sum_{i=x}^{n-1} (X_{i,i+1} - X_{i+1,i+1}) \ \\
& \geq & 0 \qquad\qquad \qquad \qquad\qquad \quad \mbox{ by (C)}.
\end{eqnarray*}
For $y>x$ we have
\begin{eqnarray*}
i_{x,y} &=& f(x,y)-f(x-1,y)-f(x,y-1)+f(x-1,y-1)\\
&=& \left(X_{1,y-x+1} + \sum_{i=1}^{x-1}X_{i+1,y-x+i+1}\right)- \left(\sum_{i=1}^{x-1}X_{i,y-x+i+1}\right)\\
&&- \left(X_{1,y-x} + \sum_{i=1}^{x-1}X_{i+1,y-x+i}\right)+ \left(\sum_{i=1}^{x-1}X_{i,y-x+i}\right)\\
& =& (X_{1,y-x+1} - X_{1,y-x}) \\
&&+ \ \ \ \ \ \ \sum_{i=1}^{x-1}( (X_{i,y-x+i}+ X_{i+1,y-x+i+1})- (X_{i,y-x+i+1} + X_{i+1,y-x+i}))\\
& \geq &  \sum_{i=1}^{x-1}( (X_{i,y-x+i}+ X_{i+1,y-x+i+1})- (X_{i,y-x+i+1} + X_{i+1,y-x+i})),  \mbox{ by (B)},\\
& \geq & 0, \qquad \qquad \qquad \qquad \qquad \qquad \qquad \qquad  \qquad \qquad \qquad \qquad \qquad \mbox{by (D).}
\end{eqnarray*}
Thus, all the parameters $i_{x,y}$ are non-negative integers. By Remark~\ref{remark_tabcond}, to verify that there is a tableau with these parameters, it now suffices to show that
$$\sum_{y=x}^{x+t}  i_{x,y} - \sum_{y=x+1}^{x+t+1}  i_{x+1,y} \geq 0\qquad  \mbox{ for all } 1 \leq x < n, \, 0 \leq t < n-x.$$
For each $1 \leq x < n$ and $0 \leq t < n-x$ we have
\begin{eqnarray}
\sum_{y=x}^{x+t} i_{x,y} &=& \left( i_{x,x} + \sum_{y=x+1}^{x+t}  i_{x,y}\right) \nonumber \\
&=& i_{x,x} + \sum_{y=x+1}^{x+t}  (f(x,y)-f(x-1,y)-f(x,y-1)+f(x-1,y-1)) \nonumber \\
&=& i_{x,x} + \sum_{y=x+1}^{x+t}  (f(x,y)-f(x-1,y)) - \sum_{y=x+1}^{x+t} ( f(x,y-1)-f(x-1,y-1)) \nonumber \\
&=& i_{x,x} + \sum_{y=x+1}^{x+t}  (f(x,y)-f(x-1,y)) - \sum_{y=x}^{x+t-1} ( f(x,y)-f(x-1,y)) \nonumber \\
&=& i_{x,x} + (f(x,x+t)-f(x-1,x+t)) -( f(x,x)-f(x-1,x)) \nonumber \\
&=& i_{x,x} + (f(x,x+t)-f(x-1,x+t)) -i_{x,x} \nonumber \\
&=& f(x,x+t)-f(x-1,x+t) \label{eqhorrible}
\end{eqnarray}
and likewise
\begin{eqnarray*}
\sum_{y=x+1}^{x+t+1} i_{x+1, y}&=& f(x+1,x+t+1)-f(x,x+t+1).
\end{eqnarray*}
Using \eqref{fact} we obtain
\begin{eqnarray*}
\sum_{y=x}^{x+t} i_{x,y} - \sum_{y=x+1}^{x+t+1} i_{x+1, y}
&=&( f(x,x+t)-f(x-1,x+t)) \\
&&- (f(x+1,x+t+1)-f(x,x+t+1))\\
&=&( f(x,x+t)-f(x+1,x+t+1)) \\
&&- (f(x-1,x+t)-f(x,x+t+1))\\
&=&( -X_{x+1, x+t+1}) - (-X_{x,x+t+1})\\
&=& X_{x,x+t+1}-X_{x+1, x+t+1},
\end{eqnarray*}
which is non-negative by (C).

Thus we have shown that there is a tableau $T$ with parameters $i_{x,y}$. 
It remains to show that $\rho(T) = X$. To this end, first let $1 \leq x < y < n$ and consider the difference
\begin{eqnarray*}
i_{x+1,y+1} - i_{x, y} =& (f(x+1,y+1) -f(x,y+1)-f(x+1,y)+f(x,y))\\
& - (f(x,y) -f(x-1,y)-f(x,y-1)+f(x-1,y-1)).
\end{eqnarray*}
Using \eqref{fact} we obtain
$$ i_{x+1,y+1} - i_{x, y} = X_{x+1,y+1} - X_{x,y+1}- X_{x+1,y} + X_{x,y},$$
which is non-negative by (D). Thus the tableau $T$ satisfies the conditions of Lemma~\ref{imentries}, so 
it will suffice to show that:
$$X_{p,q} \ = \ \sum_{j=p}^{q-1} i_{p,j} + \sum_{j=1}^{p} i_{j,q} \textrm{ for all } 1 \leq p \leq q \leq n.$$

If $p<q$,  then equation \eqref{eqhorrible} (with $x=p$ and $t = q-p-1$) gives
\begin{equation}
\sum_{j=p}^{q-1} i_{p,j} \ = \ f(p,q-1)-f(p-1,q-1). \label{eqhorrible2}
\end{equation}
Similarly,
\begin{eqnarray}
\sum_{j=1}^{p-1} i_{j,q} &=&\sum_{j=1}^{p-1}\left(f(j,q)-f(j-1,q)-f(j,q-1)+f(j-1,q-1)\right) \nonumber \\
&=&\sum_{j=1}^{p-1}\left( f(j,q) - f(j,q-1)\right) - \sum_{j=1}^{p-1}\left(f(j-1,q)-f(j-1,q-1)\right) \nonumber \\
&=&\sum_{j=1}^{p-1}\left( f(j,q) - f(j,q-1)\right) - \sum_{j=0}^{p-2}\left(f(j,q)-f(j,q-1)\right) \nonumber \\
&=&\left( f(p-1,q) - f(p-1,q-1)\right) - \left(f(0,q)-f(0,q-1)\right) \nonumber \\
&=&f(p-1,q) - f(p-1,q-1). \label{eqhorrible3}
\end{eqnarray}
If $p=q$  it is clear that $\sum_{j=p}^{p-1} i_{p,j}=0$, and applying \eqref{fact} and \eqref{eqhorrible3} we find
\begin{eqnarray*}
\rho(T)_{p,q} &=& \sum_{j=p}^{p-1} i_{p,j} + \sum_{j=1}^{p} i_{j,p} \\
&=& 0 + f(p-1,p) - f(p-1,p-1) + i_{p,p}\\
&=& f(p-1,p) - f(p-1,p-1) + f(p,p)-f(p-1,p)\\
&=& f(p,p) - f(p-1,p-1) \\
&=& X_{p,p}.
\end{eqnarray*}

If $q>p$ then using \eqref{fact}, \eqref{eqhorrible2} and \eqref{eqhorrible3} we have
\begin{eqnarray*}
\rho(T)_{p,q} &=& \sum_{j=p}^{q-1} i_{p,j} + \sum_{j=1}^{p} i_{j,q} \\
&=& (f(p,q-1)-f(p-1,q-1)) + (f(p-1,q) - f(p-1,q-1)) + i_{p,q}\\
&=& f(p,q-1)-f(p-1,q-1) + f(p-1,q) - f(p-1,q-1)\\
&& + f(p,q)-f(p-1,q)-f(p,q-1)+f(p-1,q-1)\\
&=& f(p,q) - f(p-1,q-1) \\
&=& X_{p,q}.
\end{eqnarray*}

Thus in all cases we have shown that $\rho(T)_{p,q} = X_{p,q}$, as required to conclude that $X = \rho(T)$.
\end{proof}

To establish the main theorem of this section, we combine the previous lemma with some machinery introduced by Daviaud and the present authors in \cite{Daviaud18} and further developed by Fenner and the first author \cite{JohnsonFenner} and also, in a slightly different direction, by the second author \cite{Kambites19}. We briefly recall the required definitions, simplifying them slightly where the full complexity is not required for our purposes; for a more complete introduction the reader is referred to the cited papers.

Let $w$ be a word over an alphabet $\Sigma$. For $0 \leq p < q \leq |w|+1$ and $s \in \Sigma$ we
define
$$\beta_s^w(p, q) \ = \ | \lbrace i \in \mathbb{N} \mid p < i < q, w_i = s \rbrace |$$
to be the number of occurrences of $s$ lying strictly \emph{between} $w_{p}$ and $w_{q}$.
Now let $u \in \Sigma^*$ with $|u| \leq n-1$ and let $\pi=(\pi_0, \pi_1, \ldots, \pi_{|u|})$ be a strictly increasing sequence of $|u|+1$ elements
of $[n]$. We define a formal tropical polynomial, having variables $x(s,i)$ for each letter $s \in \Sigma$ and each vertex $i \in [n]$, as follows:
\begin{eqnarray*}
f_{u, \pi}^{w} = \ \bigoplus \bigotimes_{s\in \Sigma}\bigotimes_{k=0}^{|u|} x(s, \pi_k)^{\beta_s^w(\alpha_k, \alpha_{k+1})},
\end{eqnarray*}
where the sum ranges over all $0=\alpha_0< \alpha_1<\cdots<\alpha_{|u|}<\alpha_{|u|+1}=|w|+1$ such that $w_{\alpha_k}=u_k$ for $k=1, \ldots, |u|$. Powers are interpreted tropically, and a maximum over the empty set is taken to be $-\infty$. Following \cite{JohnsonFenner} we write $f_{u}^w$ for 
$f_{u,\pi}^{w}$ where $\pi = (1, \dots, |u|+1)$. We shall write $\underline{x}$ for a vector giving values for each variable $x(s,i)$, and
$f_{u,\pi}^{w}(\underline{x})$ for the element of $\trop$ obtained by evaluating the polynomial at these values of the variables.

\begin{theorem}
\label{subvariety}
Every identity satisfied in the plactic monoid $\mathbb{P}_n$ of rank $n$ (for $n \geq 1$) is satisfied in the monoid $UT_n(\mathbb{T})$ of $n \times n$ upper triangular
tropical matrices.
\end{theorem}

\begin{proof}
We prove the contrapositive. Suppose that the identity $w=v$ (over an alphabet $\Sigma = \{a_1, \ldots, a_m\}$) does not hold in $UT_n(\mathbb{T})$. If $w$ and $v$ differ in their content (that is, the total number of times that each letter occurs) then it is trivial to see that $w=v$ cannot be an identity satisfied in $\mathbb{P}_n$ so we may assume that $w$ and $v$ have the same content, and in particular the same length.

Now by \cite[Theorem 2.2]{JohnsonFenner} there exists a word $u$ of length at most $n-1$ such that the polynomials $f_u^w$ and $f_u^v$ are not functionally equivalent. Thus we may choose values $\underline{x} \in \mathbb{R}^{mn}$ with $f_u^w(\underline{x}) \neq f_u^v(\underline{x})$.
Reasoning exactly as in the proof of \cite[Theorem 4.1]{Daviaud18}, the solution set to the system of inequalities $f_u^w(\underline{x}) \neq f_u^v(\underline{x})$ is open in the usual topology on $\mathbb{R}^{mn}$, and closed under both tropical scaling and classical scaling; from this it follows that we may assume without loss of generality that values in $\underline{x}$ are positive integers.

Notice that $u$ must occur as a scattered subword of at least one of $w$ or $v$, since otherwise we would have $f_u^w(\underline{x}) \equiv f_u^v(\underline{x}) \equiv -\infty$. Again without loss of generality, assume that $u$ is a scattered subword of $w$. 

We shall use the values $\underline{x}$ to construct a collection of $n \times n$ upper triangular tropical matrices which falsify the identity $w=v$, and which satisfy the conditions of Lemma \ref{suffinequal} and hence lie in the image of the representation $\rho$.
Since the identity fails in a homomorphic image of $\mathbb{P}_n$, we can then conclude that it fails in $\mathbb{P}_n$.

To this end, choose a positive integer $S$ larger than all the values in $\underline{x}$. Choose another positive integer $L$ very large relative to both $S$ and
all the values $f_{z,\pi}^{w}(\underline{x})$ and $f_{z,\pi}^v(\underline{x})$, as $\pi$ ranges over all increasing sequences of $|u|+1$ elements from $[n]$, and $z$ over all words of length up to $n-1$. Choose a further positive integer $G$ very large relative to $L$.
Now for $i \in [m]$ and $p,q \in [n]$ define
\begin{eqnarray*}
(A_i)_{p,q} &=& \begin{cases} -\infty & \mbox{ if } p>q\\
G + x_{i,p} & \mbox{ if } p=q\\
G + L & \mbox{ if } u_k = a_i \textrm{ for some } k \textrm{ with } p \leq k < q \textrm{ and } 1 \leq k \leq |u| \\
G + S & \mbox{ otherwise }.\end{cases}
\end{eqnarray*}
Notice that since (by assumption) $\underline{x}$ contains only positive integers, the matrices $A_i$ contain only positive integer entries on and above the diagonal. Let $\sigma : \Sigma^+ \to UT_n(\trop)$ be the unique morphism taking each $a_i$ to the matrix $A_i$.

We claim that, provided $L$ and $G$ were chosen sufficiently large, $\sigma(w)$ and $\sigma(v)$ differ in the $(1,|u|+1)$ position, and hence the matrices above falsify the identity $w=v$ in ${\rm UT}_n(\mathbb{T})$. By \cite[Lemma 5.1]{Daviaud18} we have
\begin{equation*}\label{eq1}
\sigma(w)_{1,{|u|+1}} \ = \ \bigoplus_{\substack{z \in \Sigma^*,\\ |z| \leq n-1}}\bigoplus_{\pi} \left(\bigotimes_{k=1}^{|z|} \sigma(z_k)_{\pi_{k-1}, \pi_k}\right) \otimes f_{z, \pi}^w(\underline{x}),
\end{equation*}
where $\pi$ ranges over strictly increasing sequences $(\pi_0, \pi_1, \dots, \pi_{|z|})$ of $|z|+1$ elements of $[n]$ which begin with $1$ and
end with $|u|+1$.\footnote{In the language of \cite{Daviaud18} these are ``$(|z|+1)$-vertex paths from $1$ to $|u|+1$''.} Note that if
$|z| > |u|$ then there are no such sequences $\pi$, so it suffices to consider $|z| \leq |u|$

Notice that by the definition of the $A_i$'s, each term in this maximum has the form $|z| G + i L + j S + f_{z, \pi}^w(\underline{x})$ for some
$z \in \Sigma^*$ with $|z| \leq |u|$, some $i$ and $j$ such that $i+j = |z|$ and some appropriate sequence $\pi$.  Clearly,
provided $L$ and $G$ were chosen large enough, any term with the property that $|z| = |u| = i$ and the polynomial appearing is not $-\infty$ will dominate any terms which do not have this property. It
is easy to see that there is exactly one term with this property: namely that corresponding to the choice $z = u$ and $\pi = (1, 2, \dots, |u|+1)$. Since $u$ is assumed to be a scattered subword of $w$, for this term we have
$f_{z, \pi}^w(\underline{x}) \neq -\infty$. So this term dominates all others and we have $\sigma(w)_{1,|u|+1} \ = |u| G + |u| L + f_u^w(\underline{x})$. 

If $u$ is also a scattered subword of $v$ then an identical argument yields
$\sigma(v)_{1,|u|+1} \ = |u| G = |u| L + f_u^v(\underline{x})$, and since by assumption $f_u^w(\underline{x}) \neq f_u^v(\underline{x})$ this means
$\sigma(v) \neq \sigma(w)$. On the other hand, if $u$ is not a scattered subword of $v$ then by the same argument, $v$ has terms of the
form $|z| G + i L + j S + f_{z, \pi}^v(\underline{x})$ where $i+j = |z| \leq |v| = |u|$ and the only term attaining $i = |u|$ is $-\infty$; in this case all
terms in $\sigma(v)_{1,|u|+1}$ are strictly less than the dominant term of $\sigma(w)_{1,|u|+1}$, so again $\sigma(v) \neq \sigma(w)$.

Next we claim that, again provided $L$ and $G$ are sufficiently large, each of the matrices $A_1, \ldots, A_m$ satisfy all the conditions of Lemma~\ref{suffinequal}. Let $A$ be one of these matrices. Indeed, the inequalities of
type (A) state that a sum of $n$ terms exceeds a sum of $n-1$ terms, where each term lies between $G$ and $G+L$; these must clearly
hold provided $G$ was chosen sufficiently large relative to $L$. It is clear from the constructions that (B) and (C) are satisfied.  Finally, to
see that the matrices also satisfy (D), it follows from the construction that
$(A_{i,j}+A_{i+1,j+1}) - (A_{i,j+1} + A_{i+1,j})$ is always one of $S-a$ where $a \leq S$ or $L-b$ where $b \leq L$.

Thus, by Lemma~\ref{suffinequal}, the matrices $A_1, \dots, A_m$ lie in the image of $\rho$. Since they falsify the identity $w=v$, it follows that this identity cannot hold in $\mathbb{P}_n$, as required to complete the proof.
\end{proof}

Combining Corollary \ref{cor_identities} and Theorem \ref{subvariety}, we see that the variety generated by $\mathbb{P}_n$ is bounded
below by the variety generated by $UT_n(\mathbb{T})$, and above by the variety generated by $UT_d(\mathbb{T})$ where $d=\lfloor \frac{n^2}{4}+1 \rfloor$. 
In the case $n=3$ we have $d = 3$ so these bounds coincide, giving an exact result which answers a question of Izhakian
\cite[Problem 8.3]{Izhakian19}.
\begin{corollary}\label{cor_rank3varieties}
The variety generated by $\mathbb{P}_3$ coincides with the variety generated by  
$UT_3(\mathbb{T})$.
\end{corollary}
The variety generated by $UT_3(\mathbb{T})$ is quite well-understood; for example there is a polynomial time algorithm to decide if
a given identity holds \cite{Daviaud18}, and it is known that the shortest identities satisfied have length 22 \cite{JT19}.

\begin{remark}\label{remark_rank2}
The case $n=2$ gives $d=2$ so it follows similarly that the variety generated by $\mathbb{P}_2$ coincides with that generated by $UT_2(\trop)$. 
This can also be deduced from known results. A theorem of Daviaud and the authors \cite[Theorem 4.1]{Daviaud18} is that the variety generated by $UT_2(\trop)$ coincides with that generated by the bicyclic monoid. Now $\mathbb{P}_2$ is easily seen to admit the bicyclic monoid as a quotient (since its defining relations are all consequences of the relation $12 = \epsilon$ which defines a bicyclic monoid), and so the variety it generates contains the bicyclic variety. Conversely, the work of Izhakian \cite{Izhakian19} and of Cain \textit{et al} \cite{Cain17} gives faithful representations of $\plac{2}$ inside $UT_2(\trop)$, yielding the reverse inclusion.
\end{remark}

\begin{remark}\label{remark_rank1}
In fact the same argument applies even in the rank $1$ case, but it is easily seen that $\plac{1}$ (which is a free commutative monoid of rank $1$) and $UT_1(\trop)$ (which is isomorphic, via exponentiation extended to $-\infty$ in the obvious way, to the non-negative real numbers under multiplication) both generate the variety of all commutative semigroups.
\end{remark}

When $n=4$ we have $d = 5$, so the variety of $\plac{4}$ lies somewhere between those generated by $UT_4(\trop)$ and $UT_5(\trop)$.
It is natural to ask if it is actually equal to one or other of these:
\begin{question}
Is the variety generated by $\mathbb{P}_4$ equal to that generated by $UT_4(\trop)$ and/or that generated by $UT_5(\trop)$?
\end{question}
We do not know for certain that the varieties generated by $UT_4(\trop)$ and by $UT_5(\trop)$ are actually different, but we would
be surprised if they coincide; see Conjecture~\ref{conj_successiveut} above.

\begin{remark}
In general, the gap between $n$ and $d$, and hence between our upper and lower bounds, is rather large. Of course, varieties of semigroups not being linearly ordered, there is no reason to suppose that either bound can be improved: they could both be sharp. However, it is notable that our lower bound proof considers only the (non-faithful) representation $\rho$ of $\plac{n}$, and hence establishes that
$UT_n(\trop)$ is contained not only in the variety generated by $\plac{n}$, but in the (perhaps smaller) variety generated by the image
of $\rho$. One might therefore ask whether this bound can be improved by considering more of our new faithful representation $\rho_n$.

This may be possible, but it is unlikely to be straightforward. The representation $\rho$ has the nice property that the diagonal entries are independent, in the sense that given any $n$-tuple of non-negative integers there exist elements of the image of $\rho$ with these values on the diagonal. The proof of Theorem~\ref{subvariety} makes essential use this fact, constructing matrices which falsify a given identity in large part because specific values lie on their diagonals.  Other blocks of our faithful representation do not have this property, but rather have
dependence relations between the diagonal entries. For example, in the case $n=4$, for $\{a,b\} \subset \{1, 2,3,4\}$ the diagonal entry
$\rho_n(T)_{\lbrace a, b \rbrace, \lbrace a, b \rbrace}$ counts the total number of $a$'s and $b$'s in the tableau $T$, giving rise to
obvious dependence relations such as:
$$ \rho_n(T)_{\lbrace 1, 2 \rbrace , \lbrace 1, 2 \rbrace} + 
\rho_n(T)_{\lbrace 3, 4 \rbrace , \lbrace 3, 4 \rbrace} \ = \ 
\rho_n(T)_{\lbrace 1, 3 \rbrace , \lbrace 1, 3 \rbrace} +
\rho_n(T)_{\lbrace 2, 4 \rbrace , \lbrace 2, 4 \rbrace}.$$
\end{remark}

\end{document}